\documentclass[a4paper,11pt]{article}

\usepackage[T1]{fontenc}
\usepackage[english]{babel}
\usepackage[top=25mm,bottom=25mm,left=27mm,right=27mm,headheight=16pt]{geometry}
\usepackage{amsmath,amssymb,amsthm}
\usepackage{lmodern}
\usepackage{microtype}
\usepackage{enumitem}
\usepackage{fancyhdr}
\usepackage{hyperref}
\usepackage{xurl}
\usepackage{color}

\setlist[enumerate]{leftmargin=2.2em,itemsep=0.25em,topsep=0.4em}
\allowdisplaybreaks
\numberwithin{equation}{section}

\newtheorem{theorem}{Theorem}[section]
\newtheorem{lemma}[theorem]{Lemma}
\newtheorem{corollary}[theorem]{Corollary}
\newtheorem{proposition}[theorem]{Proposition}
\theoremstyle{definition}
\newtheorem{definition}[theorem]{Definition}
\theoremstyle{remark}
\newtheorem{remark}[theorem]{Remark}

\newcommand{\F}{\mathbb{F}}
\newcommand{\ord}{\operatorname{ord}}

\title{\vspace{-1.2cm}\textbf{\boldmath Cubic AGM graphs and Hessian $3$-isogenies\\
  over finite fields $\F_q$ of odd characteristic\\
  with $q\equiv2\pmod3$}}
\author{%
  \parbox{\dimexpr\textwidth-2\tabcolsep\relax}{\centering
    {\large Yuji Hashimoto\textsuperscript{123}\qquad
    Koji Nuida\textsuperscript{23}}\\[0.6em]
    \normalsize
    \textsuperscript{1} School of Science and Engineering, Tokyo Denki University, Japan\\
    \texttt{(y.hashimoto@mail.dendai.ac.jp)}\\[0.25em]
    \textsuperscript{2} Institute of Mathematics for Industry (IMI), Kyushu University, Japan\\
    \texttt{(nuida@imi.kyushu-u.ac.jp)}\\[0.25em]
    \textsuperscript{3} Cyber Physical Security Research Institute (CPSEC),\\
    National Institute of Advanced Industrial Science and Technology (AIST), Japan
  }%
}

\date{\today}

\begin{document}
\maketitle

\begin{abstract}
Let $\F_q$ be a finite field of odd characteristic with $q\equiv2\pmod3$.
We study the directed graph defined by the Borwein--Borwein cubic arithmetic--geometric mean (AGM) over $\F_q$.
We prove that this graph is a disjoint union of directed cycles.
We associate Hessian curves with this AGM.
We also show that each edge corresponds to a $3$-isogeny defined over $\F_q$
between the curves associated with its initial and terminal vertices.
We then use a counting formula for Hessian curves to derive a lower bound for the number of cycles.
\end{abstract}

\section{Introduction}

\subsection{AGM over finite fields and its generalization}
The classical arithmetic-geometric mean (AGM) is considered on positive real numbers and is defined by $u_{i+1} = (u_i + v_i) / 2$ and $v_{i+1} = \sqrt{ u_i v_i }$.
Recently, Griffin--Ono--Saikia--Tsai \cite{GOST23} studied the AGM on finite fields $\mathbb{F}_q$ with $q \equiv 3 \pmod{4}$ and with characteristic $p \geq 7$.
They showed that the corresponding directed graph with vertices $(u_i,v_i)$ and edges $(u_i,v_i) \to (u_{i+1},v_{i+1})$, called the AGM graph, has an interesting shape named a jellyfish.
They also showed, more remarkably, that the edges of the AGM graph correspond naturally to $2$-isogenies between Legendre curves associated to the vertices of the AGM graph.
After their work, there have been subsequent studies on the AGM graph.
Kayath et al. \cite{KLNN25} considered the case of odd characteristic including the case $q \equiv 1 \pmod{4}$.
They exploited the relation with $2$-isogenies between Legendre curves mentioned above to study the structure and the number of connected components of the AGM graph.
The case $q \equiv 5 \pmod{8}$ has also been studied \cite{BG25,Vargas25}.
Moreover, for primes $p \equiv 7 \pmod{12}$, León--Muñoz \cite{LM25} derived nonlinear recurrences from chains of $2$-isogenies between supersingular elliptic curves over $\mathbb{F}_{p^2}$, and related these recurrences to the AGM.

On the other hand, a higher-degree analogue of the AGM on positive real numbers has been introduced by Borwein--Borwein \cite{BB91}.
For example, in the cubic case (instead of the classical square case), they defined
\begin{equation}
u_{i+1} = \frac{ u_i + 2 v_i }{ 3 }, \qquad
v_{i+1} = \sqrt[3]{ \frac{ v_i (u_i^2 + u_i v_i + v_i^2) }{ 3 } }\label{cubic_agm_eq}.
\end{equation}
Now we mention that their definition of such a higher-degree AGM is motivated by the theory of hypergeometric series, and whether (and how) such a higher-degree AGM can be related to the theory of elliptic curves (analogously to the relation of the quadratic AGM and Legendre curves explained above) has not been known.
In this paper, we investigate for the first time such a relation in the case of cubic AGM.
\subsection{Contributions}
In this paper, we study the Borwein--Borwein cubic AGM mentioned above for the case of finite fields.

We assume that $q=p^m\equiv2\pmod3$, where $p\ge5$ is prime and $m\ge1$ is an integer.
Every element of $\F_q$ has a unique cube root in $\F_q$.
The equations (\ref{cubic_agm_eq}) above therefore determines a unique next pair.
Set
\[
 V_{\mathrm{AGM}}
 :=\{(u,v)\in\F_q^2\mid v\ne0,\ u^3\ne v^3\}.
\]
Let $G_{\mathrm{AGM}}$ be the directed graph (digraph) with this vertex set and the update edges $(u_i, v_i) \to (u_{i+1}, v_{i+1})$.
For each vertex $(u,v)$, set $\rho:=u/v$ and associate the Hessian curve
\[
 H_{\rho}:\quad X^3+Y^3+Z^3=3\rho XYZ.
\]
The condition $u^3\ne v^3$ is equivalent to the nonsingularity of $H_{\rho}$.

Our main contributions are as follows.
\begin{enumerate}
\item \textbf{Cycle structure and $3$-isogenies.}

      The cubic AGM update $(u_i, v_i) \mapsto (u_{i+1}, v_{i+1})$ is a bijection on $V_{\mathrm{AGM}}$.
      Thus $G_{\mathrm{AGM}}$ is a disjoint union of directed cycles
      (Theorem~\ref{uni_parent_theo} and Corollary~\ref{agm_cyc_coro}).
      We use the formula of Perez Broon et al.\ for $3$-isogenies between twisted Hessian curves
      \cite[Theorem~3]{BDFM21}.
      This realizes each update edge as a $3$-isogeny over $\F_q$
      between the corresponding Hessian curves (Theorem~\ref{three_iso_agm_theo}).
\item \textbf{The quotient graph and cycle lifting.}

      Identifying vertices with the same ratio $\rho=u/v$ gives the quotient graph $G_{\rho}$ of $G_{\mathrm{AGM}}$
      , called the $\rho$-graph.
      The projection $\pi:G_{\mathrm{AGM}}\to G_{\rho}$, $\pi(u,v)=u/v$,
      is a $(q-1)$-fold covering of directed graphs.
      Let $C$ be a cycle of length $L$ in $G_{\rho}$.
      After $L$ cubic AGM updates along $C$, the second coordinate is multiplied
      by some $\kappa_C\in\F_q^\times$.
      Write $m_C$ for its multiplicative order.
      Then $\pi^{-1}(C)$ is a disjoint union of $(q-1)/m_C$ cycles,
      each of length $Lm_C$ (Theorem~\ref{rho_cover_theo}).
\item \textbf{A lower bound for the number of cycles.}

      We prove that the number $d_3(q)$ of cycles in $G_{\mathrm{AGM}}$ satisfies
      \[
       d_3(q)\ge\frac{16}{15}\sqrt q-1
      \]
      (Theorem~\ref{cyc_low_theo}).
      This bound follows from the correspondence with Hessian curves
      and the counting formula of Moody--Wu \cite{MW11} (Proposition~\ref{MW_prop}).
\end{enumerate}

\section{Preliminaries}

Let $\F_q$ be a finite field of characteristic $p\ge5$.
Write $\overline{\F}_q$ for its algebraic closure.

\subsection{Elliptic curves and isogenies}
We recall basic facts about elliptic curves and isogenies.
For details, see \cite{Galbraith12,DeFeo17,Silverman09}.

\begin{definition}[{\cite[III, Sections 1--3]{Silverman09}}]
Let $a,b\in\F_q$ satisfy $4a^3+27b^2\ne0$.
The nonsingular curve
\[
 E:\quad y^2=x^3+ax+b,
\]
together with its point at infinity $O_E$, is an elliptic curve over $\F_q$.
Its set of $\F_q$-rational points is
\[
 E(\F_q)=\{(x,y)\in\F_q^2\mid y^2=x^3+ax+b\}\cup\{O_E\}.
\]
This set is a finite abelian group under addition, with identity $O_E$.
A curve isomorphic to this curve over $\F_q$ is also called an elliptic curve.
Its identity is the point corresponding to $O_E$.
\end{definition}

\begin{definition}[{\cite[Theorem~7, Definition~24]{DeFeo17}}]
Let $E$ and $E'$ be elliptic curves over $\F_q$.
An isogeny over $\F_q$ is a nonconstant morphism of curves $\phi:E\to E'$
defined over $\F_q$ such that $\phi(O_E)=O_{E'}$.
It induces a surjective group homomorphism from $E(\overline{\F}_q)$ to $E'(\overline{\F}_q)$.
Its kernel is the finite subgroup
\[
 \ker\phi:=\{P\in E(\overline{\F}_q)\mid\phi(P)=O_{E'}\}.
\]

Let $\F_q(E)$ and $\F_q(E')$ be the respective function fields.
The pullback
\[
 \phi^*:\F_q(E')\hookrightarrow\F_q(E),\qquad
 g\longmapsto g\circ\phi
\]
makes $\F_q(E)$ a finite extension of $\phi^*(\F_q(E'))$.
The degree of $\phi$ is
\[
 \deg\phi:=[\F_q(E):\phi^*(\F_q(E'))].
\]
The isogeny is separable if this field extension is separable.
In this case, $\deg\phi=\#\ker\phi$.
\end{definition}

\begin{definition}[{\cite[Section 25.1]{Galbraith12}}]
Let $E$ and $E'$ be elliptic curves over $\F_q$, and let $\ell\ne p$ be prime.
An isogeny $\phi:E\to E'$ over $\F_q$ is called an $\ell$-isogeny if $\deg\phi=\ell$.
\end{definition}

\subsection{Twisted Hessian curves and Hessian curves}
We recall Hessian curves.
For details, see \cite{BCKL15,BDFM21,MW11}.

\begin{samepage}
\begin{definition}[{\cite[Definition~1]{BDFM21}}]
Let $a,d\in\F_q$ satisfy $a(27a-d^3)\ne0$.
The projective plane curve
\begin{equation}
 H(a,d):\quad aX^3+Y^3+Z^3=dXYZ
\end{equation}
is nonsingular and is an elliptic curve with identity $O:=(0:-1:1)$.
It is called a twisted Hessian curve.
When $a=1$, it is called a Hessian curve.
\end{definition}
\end{samepage}

\begin{proposition}[{\cite[Section 2]{BCKL15}}]
\label{hess_isom_prop}
Let $H(a,d)$ be a twisted Hessian curve over $\F_q$, and let $c\in\overline{\F}_q$ satisfy $c^3=a$.
Then
\begin{equation}
 \nu_c:H(a,d)\longrightarrow H\!\left(1,\frac dc\right),\qquad
 (X:Y:Z)\longmapsto(cX:Y:Z)
\end{equation}
is an isomorphism over $\F_q(c)$ that preserves the identity.
If $q\equiv2\pmod3$, there is a unique $c\in\F_q$ with $c^3=a$.
In this case, $H(a,d)$ and $H(1,d/c)$ are isomorphic over $\F_q$.
\end{proposition}

\subsection{3-isogenies and Hessian curves}
We recall $3$-isogenies between Hessian curves.
For details, see \cite{BDFM21,MW11}.

\begin{proposition}[{\cite[Theorem~3]{BDFM21}}]
\label{three_iso_prop}
Let $H(a,d)$ be a twisted Hessian curve over $\F_q$, and suppose that $c\in\F_q$ satisfies $c^3=a$.
Put
\[
 \begin{aligned}
 A&:=d^2c+3dc^2+9a,\qquad D:=d+6c,\\
 F_0&:=XYZ,\\
 F_1&:=c^2X^2Z+cXY^2+YZ^2,\\
 F_2&:=c^2X^2Y+cXZ^2+Y^2Z.
 \end{aligned}
\]
Then $H(A,D)$ is a twisted Hessian curve, and the map
\[
 \phi_{a,d,c}:H(a,d)\longrightarrow H(A,D),\qquad
 (X:Y:Z)\longmapsto(F_0:F_1:F_2)
\]
is a $3$-isogeny with kernel
\[
 \ker\phi_{a,d,c}
 =\langle(1:-c:0)\rangle
 =\{(0:-1:1),(1:-c:0),(1:0:-c)\}.
\]
\end{proposition}

\begin{samepage}
\begin{proposition}[{\cite[Theorem~4.2]{MW11}}]
\label{MW_prop}
Suppose that $q\equiv2\pmod3$.
The number of isogeny classes of Hessian curves over $\F_q$ is
\begin{equation}
 M_q^{\mathrm{Hess}}
 =1+2\left\lfloor\frac{2\sqrt q}{3}\right\rfloor
    -2\left\lfloor\frac{2\sqrt q}{3p}\right\rfloor.
 \label{count_eq}
\end{equation}
\end{proposition}
\end{samepage}
\section{Cubic AGM graphs over finite fields and their properties}

We define the cubic AGM graph and describe its cycle structure.

From now on, assume that $p\ge5$ and $q=p^m\equiv2\pmod3$.

\begin{definition}[Cubic AGM graph over $\F_q$]
Let
\[
 V_{\mathrm{AGM}}
 :=\{(u,v)\in\F_q^2\mid v\ne0,\ u^3\ne v^3\}.
\]
For two vertices $(u_i,v_i)$ and $(u_{i+1},v_{i+1})$,
consider the equations
\[
 u_{i+1}=\frac{u_i+2v_i}{3},
 \qquad
 v_{i+1}^3=\frac{v_i(u_i^2+u_iv_i+v_i^2)}{3}.
\]
We call these the cubic AGM update equations over $\F_q$.
We draw a directed edge
$(u_i,v_i)\to(u_{i+1},v_{i+1})$ if these equations hold.
Let $E_{\mathrm{AGM}}$ be the set of these edges.

The digraph
$G_{\mathrm{AGM}}:=(V_{\mathrm{AGM}},E_{\mathrm{AGM}})$
is called the cubic AGM graph over $\F_q$.
Its vertices and simple directed cycles are called
AGM vertices and AGM cycles, respectively.
Each edge is called an AGM update edge.
Its initial vertex is called a parent of its terminal vertex,
and its terminal vertex is called a child of its initial vertex.
\end{definition}

The existence and uniqueness of cube roots imply that each $(u_i,v_i)\in V_{\mathrm{AGM}}$
determines a unique pair $(u_{i+1},v_{i+1})\in\F_q^2$ satisfying the cubic AGM update equations.
Moreover, since $(u_i-v_i)(u_i^2+u_iv_i+v_i^2)=u_i^3-v_i^3\ne0$, we have
\begin{equation}
 \begin{aligned}
 v_{i+1}^3&=\frac{v_i(u_i^2+u_iv_i+v_i^2)}{3}\ne0,\\
 u_{i+1}^3-v_{i+1}^3&=\frac{(u_i-v_i)^3}{27}\ne0.
 \end{aligned}
 \label{agm_cond_eq}
\end{equation}
Thus the resulting pair also belongs to $V_{\mathrm{AGM}}$.
Each vertex therefore has a unique child.
This defines the cubic AGM update map
\[
 U:V_{\mathrm{AGM}}\longrightarrow V_{\mathrm{AGM}},
 \qquad
 U(u_i,v_i):=(u_{i+1},v_{i+1}),
\]
where $(u_{i+1},v_{i+1})$ is the unique pair satisfying
the cubic AGM update equations.

\begin{theorem}[Uniqueness of the parent]
\label{uni_parent_theo}
Each $(u_{i+1},v_{i+1})\in V_{\mathrm{AGM}}$ has a unique parent.
More explicitly, let $\xi_i\in\F_q$ be the unique element satisfying
$\xi_i^3=v_{i+1}^3-u_{i+1}^3$. Then the parent is given by
\begin{equation}
 (u_i,v_i)=(u_{i+1}-2\xi_i,u_{i+1}+\xi_i).
 \label{parent_formula_eq}
\end{equation}
\end{theorem}

\begin{proof}
First, suppose that a parent $(u_i,v_i)$ exists. Equation~\eqref{agm_cond_eq} gives
\[
 \left(\frac{v_i-u_i}{3}\right)^3
 =v_{i+1}^3-u_{i+1}^3=\xi_i^3.
\]
The uniqueness of cube roots gives $v_i-u_i=3\xi_i$.
Adding this to the first update equation $u_i+2v_i=3u_{i+1}$ yields
\[
 3v_i=3u_{i+1}+3\xi_i.
\]
The characteristic is not $3$, so we may divide both sides by $3$.
Substituting the result into $u_i=v_i-3\xi_i$ gives
\[
 \begin{aligned}
 v_i&=u_{i+1}+\xi_i,\\
 u_i&=(u_{i+1}+\xi_i)-3\xi_i=u_{i+1}-2\xi_i.
 \end{aligned}
\]
Thus any parent must be given by~\eqref{parent_formula_eq}.

Next, we show that the pair $(u_i,v_i)$ defined by~\eqref{parent_formula_eq} is indeed a parent.
First,
\[
 \frac{u_i+2v_i}{3}
 =\frac{(u_{i+1}-2\xi_i)+2(u_{i+1}+\xi_i)}{3}=u_{i+1},
\]
so the first update equation holds. Also,
\[
 \begin{aligned}
 u_i^2+u_iv_i+v_i^2
 &=(u_{i+1}-2\xi_i)^2+(u_{i+1}-2\xi_i)(u_{i+1}+\xi_i)
   +(u_{i+1}+\xi_i)^2\\
 &=3(u_{i+1}^2-u_{i+1}\xi_i+\xi_i^2),
 \end{aligned}
\]
and hence
\[
 \begin{aligned}
 \frac{v_i(u_i^2+u_iv_i+v_i^2)}{3}
 &=(u_{i+1}+\xi_i)
   (u_{i+1}^2-u_{i+1}\xi_i+\xi_i^2)\\
 &=u_{i+1}^3+\xi_i^3=v_{i+1}^3\ne0.
 \end{aligned}
\]
Thus the second update equation also holds, and $v_i\ne0$.
Moreover, $\xi_i^3=v_{i+1}^3-u_{i+1}^3\ne0$ implies $u_i-v_i=-3\xi_i\ne0$.
The uniqueness of cube roots then gives $u_i^3\ne v_i^3$.
Therefore $(u_i,v_i)\in V_{\mathrm{AGM}}$, and this pair is the required parent.
\end{proof}

\begin{corollary}[Cycle structure of the cubic AGM graph]
\label{agm_cyc_coro}
The graph $G_{\mathrm{AGM}}$ is a disjoint union of directed cycles.
\end{corollary}

\begin{proof}
Each AGM vertex has a unique child and, by
Theorem~\ref{uni_parent_theo}, a unique parent.
Thus $U$ is a bijection on the finite set $V_{\mathrm{AGM}}$.
Its orbits form pairwise disjoint directed cycles.
\end{proof}

\section{Cubic AGM and 3-isogenies between Hessian curves}

We relate cubic AGM updates to $3$-isogenies between Hessian curves.

\begin{proposition}[AGM vertices and nonsingular Hessian curves]
\label{hess_cond_theo}
Let $u_i,v_i\in\F_q$ with $v_i\ne0$, and set $d_i:=3u_i/v_i$.
Then $(u_i,v_i)\in V_{\mathrm{AGM}}$ if and only if
the projective plane curve
\[
 X^3+Y^3+Z^3=d_iXYZ
\]
is nonsingular.
\end{proposition}

\begin{proof}
The Hessian equation defines a nonsingular curve if and only if
$27-d_i^3\ne0$.
We have
\begin{equation}
 \begin{aligned}
 27-d_i^3
 &=27-\left(\frac{3u_i}{v_i}\right)^3\\
 &=\frac{27(v_i^3-u_i^3)}{v_i^3}.
 \end{aligned}
\end{equation}
Since $27/v_i^3\ne0$, this is equivalent to $u_i^3\ne v_i^3$,
which is precisely the condition
$(u_i,v_i)\in V_{\mathrm{AGM}}$.
\end{proof}

\begin{theorem}[$3$-isogeny corresponding to a cubic AGM update]
\label{three_iso_agm_theo}
For an AGM update edge $(u_i,v_i)\longrightarrow(u_{i+1},v_{i+1})$, set
\[
 \rho_j:=\frac{u_j}{v_j},\qquad
 H_{\rho_j}:=H(1,3\rho_j)
 \quad(j=i,i+1)
\]
and define
\[
 A_i:=9(\rho_i^2+\rho_i+1),\qquad
 D_i:=3(\rho_i+2),\qquad
 \eta_i:=\frac{3v_{i+1}}{v_i}\in\F_q^\times.
\]
Then
\[
 \eta_i^3=A_i,\qquad \frac{D_i}{\eta_i}=3\rho_{i+1}.
\]
Moreover, the map
\begin{equation}
 \begin{aligned}
 \phi_i:H_{\rho_i}&\longrightarrow H_{\rho_{i+1}},\\
 (X:Y:Z)&\longmapsto
 \bigl(\eta_iXYZ:
 X^2Z+XY^2+YZ^2:
 X^2Y+XZ^2+Y^2Z\bigr)
 \end{aligned}
 \label{agm_hess_eq}
\end{equation}
is a $3$-isogeny.
\end{theorem}

\begin{proof}

By Proposition~\ref{hess_cond_theo}, the curves
$H_{\rho_i}$ and $H_{\rho_{i+1}}$ are nonsingular.
Substituting $a=c=1$ and $d=3\rho_i$
into Proposition~\ref{three_iso_prop} gives
\[
 \begin{aligned}
 A&=(3\rho_i)^2+3(3\rho_i)+9
   =9(\rho_i^2+\rho_i+1)=A_i,\\
 D&=3\rho_i+6=3(\rho_i+2)=D_i.
 \end{aligned}
\]
We thus obtain the $3$-isogeny
\[
 \begin{aligned}
 \psi_i:H_{\rho_i}&\longrightarrow H(A_i,D_i),\\
 (X:Y:Z)&\longmapsto
 \bigl(XYZ:X^2Z+XY^2+YZ^2:X^2Y+XZ^2+Y^2Z\bigr).
 \end{aligned}
\]

Next, we show that $H(A_i,D_i)$ and $H_{\rho_{i+1}}$ are isomorphic.
The cubic AGM update equations give
\[
 \begin{aligned}
 \eta_i^3
 &=\frac{27v_{i+1}^3}{v_i^3}
  =\frac{9(u_i^2+u_iv_i+v_i^2)}{v_i^2}
  =9(\rho_i^2+\rho_i+1)=A_i,\\
 \frac{D_i}{\eta_i}
 &=\frac{3(\rho_i+2)}{3v_{i+1}/v_i}
  =\frac{u_i+2v_i}{v_{i+1}}
  =\frac{3u_{i+1}}{v_{i+1}}=3\rho_{i+1}.
 \end{aligned}
\]
Since $\eta_i\in\F_q^\times$, Proposition~\ref{hess_isom_prop} gives the following
isomorphism over $\F_q$ that preserves the identity element:
\[
 \begin{aligned}
 \nu_{\eta_i}:H(A_i,D_i)&\longrightarrow
 H\!\left(1,\frac{D_i}{\eta_i}\right)=H_{\rho_{i+1}},\\
 (U:V:W)&\longmapsto(\eta_iU:V:W).
 \end{aligned}
\]
The composition $\nu_{\eta_i}\circ\psi_i$ is the map $\phi_i$ in~\eqref{agm_hess_eq}.
Since $\nu_{\eta_i}$ is an isomorphism, $\phi_i$ is a $3$-isogeny.
\end{proof}

\begin{remark}[Reason that $u_i\ne0$ is not required]
Let $v_i\in\F_q^\times$. Then $(0,v_i)\in V_{\mathrm{AGM}}$, and the corresponding
Hessian curve $H(1,0):X^3+Y^3+Z^3=0$ is nonsingular.
Also, $(-2v_i)^3-v_i^3=-9v_i^3\ne0$, so $(-2v_i,v_i)$ is a vertex.
Substituting $u_i=-2v_i$ into the AGM update equations gives
\[
 \begin{aligned}
 u_{i+1}&=\frac{-2v_i+2v_i}{3}=0,\\
 v_{i+1}^3&=\frac{v_i(4v_i^2-2v_i^2+v_i^2)}{3}=v_i^3.
 \end{aligned}
\]
The uniqueness of cube roots gives $v_{i+1}=v_i$. Hence
\[
 (-2v_i,v_i)\longrightarrow(0,v_i)
\]
is an AGM update edge.
If we further restricted $V_{\mathrm{AGM}}$ by requiring
$u_i\ne0$, the resulting set would not be closed under $U$.
\end{remark}

\section{The quotient of the cubic AGM graph and cycle lifting}
We describe AGM cycles using a quotient graph defined
by the ratio $u/v$.

\begin{proposition}[A product description of the AGM vertex set]
\label{agm_rho_bij_theo}
Set $V_{\rho}:=\F_q\setminus\{1\}$.
The maps
\[
 \begin{aligned}
 \Phi &:V_{\mathrm{AGM}}\longrightarrow
        V_{\rho}\times\F_q^\times,
 &\Phi(u,v)&=(u/v,v),\\
 \Psi &:V_{\rho}\times\F_q^\times\longrightarrow
        V_{\mathrm{AGM}},
 &\Psi(\rho,v)&=(\rho v,v)
 \end{aligned}
\]
are well-defined and satisfy
\[
 \Psi\circ\Phi=\operatorname{id}_{V_{\mathrm{AGM}}},
 \qquad
 \Phi\circ\Psi=\operatorname{id}_{V_{\rho}\times\F_q^\times},
\]
where $\operatorname{id}_{V_{\mathrm{AGM}}}$ and
$\operatorname{id}_{V_{\rho}\times\F_q^\times}$ denote
the identity maps on $V_{\mathrm{AGM}}$ and
$V_{\rho}\times\F_q^\times$, respectively.
\end{proposition}

\begin{proof}
If $(u,v)\in V_{\mathrm{AGM}}$, then $v\ne0$ and $(u/v)^3=u^3/v^3\ne1$.
Thus $u/v\in V_{\rho}$, so $\Phi(u,v)\in V_{\rho}\times\F_q^\times$.

Conversely, let $(\rho,v)\in V_{\rho}\times\F_q^\times$.
Since $q\equiv2\pmod3$, the only cube root of $1$ in $\F_q$ is $1$.
Hence $\rho\ne1$ implies $\rho^3\ne1$, and therefore
\[
 (\rho v)^3-v^3=(\rho^3-1)v^3\ne0.
\]
Together with $v\ne0$, this gives $\Psi(\rho,v)\in V_{\mathrm{AGM}}$.
The two compositions are
\[
 \begin{aligned}
 (\Psi\circ\Phi)(u,v)
 &=\Psi(u/v,v)=\bigl((u/v)v,v\bigr)=(u,v),\\
 (\Phi\circ\Psi)(\rho,v)
 &=\Phi(\rho v,v)=\bigl((\rho v)/v,v\bigr)=(\rho,v),
 \end{aligned}
\]
which proves the claim.
\end{proof}

We use $\Phi$ to represent each $(u,v)\in V_{\mathrm{AGM}}$ by the pair $(\rho, v)$ of the ratio $\rho = u / v$ and the second coordinate $v$.
The corresponding vertex in the original coordinates is $\Psi(\rho,v)=(\rho v,v)$.

For $\rho\in V_{\rho}$, we have
$(\rho-1)(\rho^2+\rho+1)=\rho^3-1\ne0$.
Hence there is a unique $\vartheta(\rho)\in\F_q^\times$ satisfying
\[
 \vartheta(\rho)^3=\frac{\rho^2+\rho+1}{3}.
\]
Substituting $u_i=\rho_i v_i$ into the AGM update gives
\[
 \begin{aligned}
 u_{i+1}&=\frac{\rho_i v_i+2v_i}{3}
          =\frac{\rho_i+2}{3}v_i,\\
 v_{i+1}^3&=\frac{v_i(\rho_i^2v_i^2+\rho_iv_i^2+v_i^2)}{3}
          =\vartheta(\rho_i)^3v_i^3.
 \end{aligned}
\]
Uniqueness of cube roots and $v_i\ne0$ give
\[
 v_{i+1}=\vartheta(\rho_i)v_i,\qquad
 \rho_{i+1}=\frac{(\rho_i+2)v_i}{3\vartheta(\rho_i)v_i}
           =\frac{\rho_i+2}{3\vartheta(\rho_i)}.
\]
Set
$f(\rho):=\frac{\rho+2}{3\vartheta(\rho)}\,\,(\rho\in V_{\rho}).$
The preceding calculation shows that an AGM update sends
a vertex with ratio $\rho$ to a vertex with ratio $f(\rho)$.
Since the updated vertex belongs to $V_{\mathrm{AGM}}$,
we have $f(\rho)\in V_{\rho}$.
We call $f:V_{\rho}\to V_{\rho}$ the ratio update map.
The cubic AGM update in these coordinates is
$T:=\Phi\circ U\circ\Psi$.
The preceding calculation gives
\begin{equation}
 T:V_{\rho}\times\F_q^\times
 \longrightarrow V_{\rho}\times\F_q^\times,
 \qquad
 T(\rho,v)=\bigl(f(\rho),\vartheta(\rho)v\bigr).
 \label{agm_scalar_up_eq}
\end{equation}
In particular,
\[
 U\circ\Psi=\Psi\circ T.
\]

\begin{definition}[$\rho$-graph]
The digraph $G_{\rho}:=(V_{\rho},E_{\rho})$ is called the $\rho$-graph, where
\[
 E_{\rho}:=\{(\rho,f(\rho))\mid\rho\in V_{\rho}\}.
\]
\end{definition}

Two vertices $(u,v)$ and $(\widetilde u,\widetilde v)$
with the same ratio are related by scalar multiplication
with $\gamma=\widetilde v/v\in\F_q^\times$:
$(\widetilde u,\widetilde v)=(\gamma u,\gamma v)$.
By~\eqref{agm_scalar_up_eq}, the ratio after an AGM update
is $f(\rho)$, which depends only on the initial ratio $\rho$.
Thus $G_{\rho}$ is the quotient of $G_{\mathrm{AGM}}$
obtained by identifying vertices with the same ratio.

\begin{theorem}[Bijectivity of $f$]
\label{rho_up_theo}
The map $f:V_{\rho}\to V_{\rho}$ is bijective.
\end{theorem}

\begin{proof}
For any $\eta\in V_{\rho}$, Theorem~\ref{uni_parent_theo} gives
a parent $(u,v)\in V_{\mathrm{AGM}}$ of $(\eta,1)\in V_{\mathrm{AGM}}$.
Set $\rho:=u/v\in V_{\rho}$.
In the ratio and second coordinate, this update edge is expressed as
\[
 T(\rho,v)=\bigl(f(\rho),\vartheta(\rho)v\bigr)=(\eta,1)
\]
by~\eqref{agm_scalar_up_eq}. Comparing the first coordinates gives $f(\rho)=\eta$.
Since $\eta$ was arbitrary, $f$ is surjective. Then, as $V_{\rho}$ is finite, $f$ is a bijection.
\end{proof}

\begin{corollary}[The inverse of the AGM update]
\label{three_agm_inv_theo}
The map $T$ in~\eqref{agm_scalar_up_eq} is bijective.
For $(\rho,v)\in V_{\rho}\times\F_q^\times$, set $\sigma:=f^{-1}(\rho)$.
Then
\begin{equation}
 T^{-1}(\rho,v)=\left(\sigma,\frac{v}{\vartheta(\sigma)}\right).
 \label{agm_scalar_inv_eq}
\end{equation}
\end{corollary}

\begin{proof}
The map $U$ is bijective by Theorem~\ref{uni_parent_theo}.
The maps $\Phi$ and $\Psi$ are bijective
by Proposition~\ref{agm_rho_bij_theo}.
Thus $T=\Phi\circ U\circ\Psi$ is bijective.

Let $(\rho,v)\in V_{\rho}\times\F_q^\times$ and
$\sigma=f^{-1}(\rho)$.
Since $v,\vartheta(\sigma)\ne0$, the pair
$\bigl(\sigma,v/\vartheta(\sigma)\bigr)$ belongs to
$V_{\rho}\times\F_q^\times$.
Moreover,
\[
 \begin{aligned}
 T\left(\sigma,\frac{v}{\vartheta(\sigma)}\right)
 &=\left(
     f(\sigma),
     \vartheta(\sigma)\frac{v}{\vartheta(\sigma)}
   \right)\\
 &=(\rho,v).
 \end{aligned}
\]
\end{proof}

\begin{proposition}[The AGM update commutes with scalar multiplication]
For $\gamma\in\F_q^\times$, define
\[
 S_\gamma:V_{\rho}\times\F_q^\times
 \longrightarrow V_{\rho}\times\F_q^\times,
 \qquad
 S_\gamma(\rho,v):=(\rho,\gamma v).
\]
Under $\Psi$, this map corresponds to scalar multiplication
$(u,v)\mapsto(\gamma u,\gamma v)$.
Moreover,
\[
 T\circ S_\gamma=S_\gamma\circ T.
\]
\end{proposition}

\begin{proof}
For every $(\rho,v)\in V_{\rho}\times\F_q^\times$, we have
\[
 \begin{aligned}
 \Psi\bigl(S_\gamma(\rho,v)\bigr)
 &=\Psi(\rho,\gamma v)\\
 &=(\rho\gamma v,\gamma v)\\
 &=\gamma(\rho v,v)
 =\gamma\Psi(\rho,v).
 \end{aligned}
\]
This proves the correspondence with scalar multiplication.
Also,
\[
 \begin{aligned}
 T\bigl(S_\gamma(\rho,v)\bigr)
 &=\bigl(f(\rho),\vartheta(\rho)\gamma v\bigr),\\
 S_\gamma\bigl(T(\rho,v)\bigr)
 &=\bigl(f(\rho),\gamma\vartheta(\rho)v\bigr).
 \end{aligned}
\]
These expressions agree by commutativity of multiplication.
\end{proof}

For a finite digraph $G$, let $\operatorname{Cyc}(G)$ denote the set of its simple directed cycles.
Cycles that differ only in their starting vertex are identified, and loops are included as cycles of length $1$.

\begin{theorem}[The quotient by ratios and cycle lifting]
\label{rho_cover_theo}
For the map on vertices in the original coordinates,
\[
 \pi:V_{\mathrm{AGM}}\longrightarrow V_{\rho},\qquad
 (u,v)\longmapsto u/v,
\]
the following hold.
\begin{enumerate}
 \item By sending each AGM update edge $(u,v)\to(u',v')$ to
 $\pi(u,v)\to\pi(u',v')$, the map $\pi$ defines a $(q-1)$-fold covering of digraphs
 $\pi:G_{\mathrm{AGM}}\to G_{\rho}$.
 For each $\rho\in V_{\rho}$, its fiber is
 \[
  \pi^{-1}(\rho)=\{(\rho v,v)\mid v\in\F_q^\times\}.
 \]
 Here $(q-1)$-fold means that each vertex has exactly $q-1$ vertices in its fiber.
 \item For a cycle
 $C=(\rho_0\to\rho_1\to\cdots\to\rho_{L-1}\to\rho_0)\in\operatorname{Cyc}(G_{\rho})$
 of length $L$, put
 \[
  \kappa_C:=\prod_{i=0}^{L-1}\vartheta(\rho_i),\qquad
  m_C:=\ord_{\F_q^\times}(\kappa_C).
 \]
 Here $m_C$ is the least positive integer satisfying $\kappa_C^{m_C}=1$.
 The multiplier $\kappa_C$ does not depend on the starting vertex of $C$.
 The subdigraph $\pi^{-1}(C)$ consisting of the preimages of the vertices and edges of $C$
 is a disjoint union of $(q-1)/m_C$ cycles of length $Lm_C$.
\end{enumerate}
\end{theorem}

\begin{proof}
We first prove Part~1.
By Proposition~\ref{agm_rho_bij_theo}, every AGM vertex
has a unique expression $\Psi(\rho,v)=(\rho v,v)$.
Since
\[
 \pi\bigl(\Psi(\rho,v)\bigr)=\pi(\rho v,v)=\rho,
\]
the map $v\mapsto(\rho v,v)$ is a bijection from
$\F_q^\times$ to $\pi^{-1}(\rho)$.
Thus
\[
 \pi^{-1}(\rho)=\{(\rho v,v)\mid v\in\F_q^\times\},
\]
and every fiber has $q-1$ vertices.
In particular, $\pi$ is surjective.

Let $x=\Psi(\rho,v)$.
Using $U\circ\Psi=\Psi\circ T$
and~\eqref{agm_scalar_up_eq}, we obtain
\[
 \begin{aligned}
 \pi\bigl(U(x)\bigr)
 &=\pi\bigl(\Psi(f(\rho),\vartheta(\rho)v)\bigr)\\
 &=f(\rho)
 =f\bigl(\pi(x)\bigr).
 \end{aligned}
\]
Hence $\pi\circ U=f\circ\pi$.

The maps $U$ and $f$ are bijective by
Theorems~\ref{uni_parent_theo} and~\ref{rho_up_theo}.
For each $x\in V_{\mathrm{AGM}}$, the unique outgoing
edge at $x$ is $x\to U(x)$.
The equality above shows that its image is
$\pi(x)\longrightarrow f(\pi(x))$,
the unique outgoing edge at $\pi(x)$ in $G_{\rho}$.
Thus $\pi$ maps the outgoing edge set at $x$
bijectively onto the outgoing edge set at $\pi(x)$.

For incoming edges, applying the same equality to
$U^{-1}(x)$ gives
\[
 \begin{aligned}
 f\bigl(\pi(U^{-1}(x))\bigr)
 &=\pi\bigl(U(U^{-1}(x))\bigr)\\
 &=\pi(x).
 \end{aligned}
\]
Since $f$ is bijective, this implies
\[
 \pi\bigl(U^{-1}(x)\bigr)=f^{-1}\bigl(\pi(x)\bigr).
\]
Therefore the unique incoming edge $U^{-1}(x)\to x$
is mapped to $f^{-1}(\pi(x))\longrightarrow\pi(x)$,
the unique incoming edge at $\pi(x)$ in $G_{\rho}$.
Thus $\pi$ also maps the incoming edge set at $x$
bijectively onto the incoming edge set at $\pi(x)$.

These two bijections establish the covering property.
Since every fiber has $q-1$ vertices,
$\pi$ is a $(q-1)$-fold covering of digraphs.

We next prove Part~2.
By commutativity of multiplication, $\kappa_C$ does not
depend on the starting vertex of $C$.
Every vertex $x$ of $\pi^{-1}(C)$ has a unique expression
\[
 x=\Psi(\rho_j,v),
 \qquad
 j\in\{0,\ldots,L-1\},\quad v\in\F_q^\times.
\]
For a positive integer $n$, the identity
$U\circ\Psi=\Psi\circ T$ gives
\[
 U^n(x)=\Psi\bigl(T^n(\rho_j,v)\bigr).
\]
Since $\Psi$ is injective,
\[
 U^n(x)=x
 \quad\Longleftrightarrow\quad
 T^n(\rho_j,v)=(\rho_j,v).
\]
By~\eqref{agm_scalar_up_eq}, the first coordinate of
$T^n(\rho_j,v)$ is $f^n(\rho_j)$.
The least period of $\rho_j$ under $f$ is $L$, so
\[
 f^n(\rho_j)=\rho_j
 \quad\Longleftrightarrow\quad
 L\mid n.
\]
Thus $U^n(x)=x$ can hold only when $n=kL$
for some positive integer $k$.

Iterating~\eqref{agm_scalar_up_eq} along $C$ gives
\[
 \begin{aligned}
 T^L(\rho_j,v)
 &=\left(
     \rho_j,
     \left(\prod_{i=0}^{L-1}\vartheta(\rho_i)\right)v
   \right)
 =(\rho_j,\kappa_Cv),\\
 T^{kL}(\rho_j,v)
 &=(\rho_j,\kappa_C^kv)
 \qquad(k\ge1).
 \end{aligned}
\]
Since $v\ne0$,
\[
 T^{kL}(\rho_j,v)=(\rho_j,v)
 \quad\Longleftrightarrow\quad
 \kappa_C^kv=v
 \quad\Longleftrightarrow\quad
 \kappa_C^k=1.
\]
The least positive integer $k$ satisfying this condition
is $m_C$.
Thus every vertex of $\pi^{-1}(C)$ has least period
$Lm_C$ under $U$.
It follows that $\pi^{-1}(C)$ is a disjoint union
of cycles of length $Lm_C$.
By Part~1, this subdigraph has $L(q-1)$ vertices.
Hence the number of cycles is
\[
 \frac{L(q-1)}{Lm_C}=\frac{q-1}{m_C}.
\]
\end{proof}

\section{Lower bounds for the number of cycles}

We derive a lower bound for the number of AGM cycles using the counting formula in Proposition~\ref{MW_prop}.

Let $d_3(q):=\#\operatorname{Cyc}(G_{\mathrm{AGM}}), c_{\rho}(q):=\#\operatorname{Cyc}(G_{\rho})$
be the numbers of cycles in the AGM graph and the $\rho$-graph, respectively.
Let $M_q^{\mathrm{Hess}}$ be the number defined in Proposition~\ref{MW_prop}.

\begin{samepage}
\begin{lemma}[Counting cycles and Hessian curves]
\label{cyc_count_lem}
For each $C\in\operatorname{Cyc}(G_{\rho})$, let $m_C$ be as in Theorem~\ref{rho_cover_theo}.
Then
\begin{equation}
 d_3(q)
 =\sum_{C\in\operatorname{Cyc}(G_{\rho})}\frac{q-1}{m_C}
 \ge c_{\rho}(q)
 \ge M_q^{\mathrm{Hess}}.
\end{equation}
\end{lemma}
\end{samepage}

\begin{proof}
By Theorem~\ref{rho_up_theo}, $G_{\rho}$ is a disjoint union of directed cycles.
Hence the inverse images $\pi^{-1}(C)$ of these cycles partition $G_{\mathrm{AGM}}$.
By Part~2 of Theorem~\ref{rho_cover_theo}, each inverse image consists of $(q-1)/m_C$ AGM cycles.
Also, $m_C$ is the order of an element of $\F_q^\times$, so $m_C\mid(q-1)$.
Therefore
\[
 d_3(q)
 =\sum_{C\in\operatorname{Cyc}(G_{\rho})}\frac{q-1}{m_C}
 \ge\sum_{C\in\operatorname{Cyc}(G_{\rho})}1
 =c_{\rho}(q).
\]

Next, set $M:=M_q^{\mathrm{Hess}}$.
By definition, we can choose $M$ Hessian curves $H(1,d_1),\ldots,H(1,d_M)$
that are pairwise non-isogenous over $\F_q$.
For each $j=1,\ldots,M$, set $\rho_j:=d_j/3$.
Since $H(1,d_j)$ is nonsingular, we have $d_j^3\ne27$, and hence
\[
 \rho_j^3-1
 =\left(\frac{d_j}{3}\right)^3-1
 =\frac{d_j^3-27}{27}\ne0.
\]
Thus $\rho_j\ne1$, so $\rho_j\in V_{\rho}$.
We can therefore choose the cycle $C_j$ containing this vertex.

By Part~1 of Theorem~\ref{rho_cover_theo}
and Theorem~\ref{three_iso_agm_theo},
the Hessian curves corresponding to the endpoints of
each edge of $G_{\rho}$ are connected by a $3$-isogeny
defined over $\F_q$.
If $C_i=C_j$ for some $i\ne j$, composing the isogenies along this cycle gives an isogeny
\[
 H(1,d_i)=H_{\rho_i}\longrightarrow H_{\rho_j}=H(1,d_j)
\]
defined over $\F_q$. This contradicts the choice of the curves.
Thus $C_1,\ldots,C_M$ are distinct, and we obtain
\[
 c_{\rho}(q)\ge M=M_q^{\mathrm{Hess}}.
\]
\end{proof}

\begin{lemma}[A lower bound from the counting formula]
\label{agm_cyc_low_lem}
We have
\begin{equation}
 d_3(q)\ge
 1+2\left\lfloor\frac{2\sqrt q}{3}\right\rfloor
 -2\left\lfloor\frac{2\sqrt q}{3p}\right\rfloor.
 \label{agm_cyc_low_eq}
\end{equation}
\end{lemma}

\begin{proof}
By Lemma~\ref{cyc_count_lem}, $d_3(q)\ge M_q^{\mathrm{Hess}}$.
Substituting the counting formula~\eqref{count_eq} from Proposition~\ref{MW_prop}
on the right gives the result.
\end{proof}

\begin{theorem}[A lower bound for the number of cubic AGM cycles]
\label{cyc_low_theo}
We have
\[
 d_3(q)\ge\frac{16}{15}\sqrt q-1.
\]
\end{theorem}

\begin{proof}
Applying the inequalities $x-1\le\lfloor x\rfloor\le x$ and $p\ge5$
to~\eqref{agm_cyc_low_eq} in Lemma~\ref{agm_cyc_low_lem}, we obtain
\[
 \begin{aligned}
 d_3(q)
 &\ge1+2\left(\frac{2\sqrt q}{3}-1\right)-\frac{4\sqrt q}{3p}\\
 &=\frac{4}{3}\left(1-\frac1p\right)\sqrt q-1\\
 &\ge\frac{16}{15}\sqrt q-1.
 \end{aligned}
\]
\end{proof}

\paragraph*{Acknowledgements.}
This work was supported by JSPS KAKENHI Grant Numbers JP24K17281 and JP25K14994, Japan.

\bibliographystyle{abbrvurl}
\bibliography{references}

\end{document}